\documentclass[a4paper,10pt]{amsart}

\usepackage{amsrefs}
\usepackage{amssymb}
\usepackage{amsmath}
\usepackage{amsthm}
\usepackage{color}
\usepackage{enumerate}
\input arrow

\renewcommand{\ker}{\textrm{Ker }} \renewcommand{\hom}{\textrm{Hom}}
\DeclareMathOperator{\ext}{Ext}

\newcommand{\modr}{\textrm{Mod-}R}
\newcommand{\rmod}{R\textrm{-Mod}}

\newcommand{\cf}{\textrm{cf}}

\newcommand{\br}{\textrm{Br}}
\newcommand{\supp}{\textrm{supp}}
\newcommand{\rest}{\upharpoonright}

\newtheorem{definition}{Definition}[section]
\newtheorem{proposition}[definition]{Proposition}
\newtheorem{theorem}[definition]{Theorem}
\newtheorem{example}[definition]{Example}
\newtheorem{examples}[definition]{Examples}
\newtheorem{lemma}[definition]{Lemma}
\newtheorem{corollary}[definition]{Corollary}
\newtheorem{remark}[definition]{Remark}

\title[THE COTORSION PAIR GENERATED BY FLAT MITTAG-LEFFLER MODULES]{THE COTORSION PAIR GENERATED BY THE CLASS OF FLAT MITTAG-LEFFLER MODULES}

\author{Manuel Cort\'es-Izurdiaga} \address{Department of Mathematics,
  University of Almeria, E-04071, Almeria, Spain} \date{\today}
\thanks{The author is very grateful to J. \v{S}aroch for many
  interesting discussions on the subject and, specially, for
  the proof of Proposition \ref{p:AlephOmegaFlats}. Partially supported by research project MTM-2014-54439 and by
  research group ``Categor\'{\i}as, computaci\'on y teor\'{\i}a de
  anillos'' (FQM211) of the University of Almer\'{\i}a}
\email{mizurdia@ual.es} \keywords{Flat Mittag-Leffler module;
  cotorsion pair; totally ordered direct limit of projectives}
\subjclass[2000]{16D40, 16E99}

\begin{document}

\maketitle

\begin{abstract}
  Let $R$ be a ring and denote by $\mathcal{FM}$ the class of all flat
  and Mittag-Leffler left $R$-modules. In \cite{BazzoniStovicek2} it
  is proved that, if $R$ is countable, the orthogonal class of
  $\mathcal{FM}$ consists of all cotorsion modules. In this note we
  extend this result to the class of all rings $R$ satisfying that
  each flat left $R$-module is filtered by totally ordered direct limits of
  projective modules. This class of rings contains all countable, left
  perfect and discrete valuation domains. Moreover, assuming that
  there do not exist inaccessible cardinals, we obtain that, over
  these rings, all flat left $R$-modules have finite projective
  dimension.
\end{abstract}

\section*{INTRODUCTION}
\label{sec:introduction}

Let $R$ be an associative ring with unit. A left $R$-module $M$ is
said to be Mittag-Leffler if for each family $\{M_i:i \in I\}$ of
right $R$-modules, the canonical morphism from $\left(\prod_{i \in
    I}M_i\right) \otimes_R M$ to $\prod_{i \in I}M_i \otimes_RM$ is
monic. Mittag-Leffler modules were introduced in \cite{RaynaudGruson}
and, in the last few years,
they have been brought into play in different settings, such
as to solve the Baer splitting problem raised by Kaplansky in 1962 in
\cite{Kaplansky} (see \cite{AngeleriBazzoniHerbera}), or in proving
that each tilting class is determined by a class of finitely presented
modules (see \cite{BazzoniStovicek}). Moreover, in
\cite[p. 266]{Drinfeld}, Drinfeld has proposed quasi-coherent sheaves
whose sections at affine open sets are flat and Mittag-Leffler modules
as the appropriate objects defining infinite-dimensional vector
bundles on a scheme (these sheaves are called Drinfeld vector bundles
in \cite{EstradaGuilPrestTrlifaj}). Since then, many authors have
studied the homological and homotopical properties of the class of all
flat and Mittag-Leffler modules.

In this paper we are interested is the cotorsion pair generated by the
class $\mathcal{FM}$ of all flat and Mittag-Leffler left $R$-modules,
that is, with $({^\perp}(\mathcal{FM}^\perp),\mathcal{FM}^\perp)$
(where, for each class of left $R$-modules $\mathcal X$,
${^\perp}\mathcal X$ and $\mathcal X^{\perp}$ are the left and right
orthogonal classes of $\mathcal X$ with respect to the
$Ext$-functor). In \cite{BazzoniStovicek2} it is proved that if $R$ is
a countable ring, this cotorsion pair is precisely the Enochs
cotorsion pair, i. e., the one generated by the flat modules. The key
of the proof is Theorem 6 (a) which states that, for any ring, any
countable totally ordered direct limit of flat Mittag-Leffler left
$R$-modules belongs to the double orthogonal class
${^\perp}(\mathcal{FM}^\perp)$ (this result extends \cite[Theorem
3.8]{SarochTrlifaj}, which gives another proof using
the Singular Cardinal Hypothesis). In this paper we prove that this
result is true for each (not necessarily countable) totally ordered
direct limit of flat Mittag-Leffler left $R$-modules by extending the construction
given in \cite{SlavikTrlifaj}. As a consequence, we obtain that
$({^\perp}(\mathcal{FM}^\perp),\mathcal{FM}^\perp)$ is the Enochs
cotorsion pair for those rings in which each flat left $R$-module is
filtered by totally ordered direct limits of projective modules.

This last condition is treated in Section 3. The Govorov-Lazard theorem states that each flat module is a direct
limit of projective modules, see \cite{Govorov} and
\cite{Lazard1964}. But, when is each flat module a totally ordered
direct limit of projective modules? That is, if $F$ is a flat left
$R$-module, when does there exist a direct system of projective
modules, $(P_i,f_{ij})_{i,j \in I}$, such that $I$ is totally ordered
and $\displaystyle F \cong \lim_{\longrightarrow}P_i$? More generally:
which are the rings for which each flat module is filtered by totally
ordered direct limits of projective modules?. We devote Section 3 to study
this class of rings, which contains all countable, left perfect and
discrete valuation rings. We prove that there exist rings not
satisfying this property (Example \ref{e:example}) and, assuming that
there do not exist inaccessible cardinals, we prove that, for any ring
$R$ satisfying this property, there exists a natural number $n$ such
that $R$ is left $n$-perfect, that is, each flat left $R$-module has
projective dimension less than or equal to $n$ (Corollary
\ref{c:n-perfect}).

\section{PRELIMINARIES}
\label{sec:preliminaries}

For any set $A$ we shall denote by $|A|$ its cardinality. Given a map
$f:A \rightarrow B$ and a subset $A' \subseteq A$, $f \rest A'$ will
be the restriction of $f$ to $A'$. For any cardinal $\lambda$, we
shall denote its cofinality by $\cf(\lambda)$ and by $\lambda^+$ the
next cardinal of $\lambda$. An infinite cardinal $\lambda$ is said to
be inaccessible if $\lambda$ is limit and regular. The existence of
inaccessible cardinals cannot be proved in ZFC; moreover, the
existence of inaccessible cardinals is relatively consistent with ZFC.

We fix, for the rest of the paper, a non necessarily commutative ring
with identity $R$; module will mean left $R$-module. Given any
family of modules, $\{M_i:i \in I\}$, and $x \in \prod_{i \in I}M_i$, we
shall denote by $x(i)$ the ith-coordinate of $x$ for each $i \in I$
and by $\supp(x) = \{i \in I:x(i) \neq 0\}$. If $\lambda$ is an
infinite cardinal, we shall denote by $\prod_{i \in I}^\lambda M_i$
the $\lambda$-product of the family $\{M_i:i \in I\}$, that is,
\begin{displaymath}
  \left\{x \in \prod_{i \in I}M_i:|\supp(x)| < \lambda\right\}.
\end{displaymath}
Let $\mathcal X$ be a class of modules, $(X_i,f_{ij})_{i, j \in I}$ a
direct of system of modules belonging to $\mathcal X$ and $M$ its
direct limit. If $I$ is totally ordered (resp. well ordered), we shall say
that $M$ is a totally ordered (resp. well ordered) direct limit of modules
belonging to $\mathcal X$. A $\mathcal X$-filtration of a module $N$
is a continuous chain of submodules of $N$,
$\{N_\alpha:\alpha < \kappa\}$, where $\kappa$ is an ordinal, such
that $N = \bigcup_{\alpha < \kappa}N_\alpha$ and
$\frac{N_{\alpha+1}}{N_\alpha} \in \mathcal X$ for each
$\alpha < \kappa$. In this case, we shall say that $N$ is
$\mathcal X$-filtered. Concerning continuous chains of direct
summands, we shall use the following well known lemma (see, for
instance, \cite[Lemme (3.1.2)]{RaynaudGruson} or the proof of \cite[Lemma 3.3]{GuilIzurdiagaTorrecillas2}):

\begin{lemma}\label{l:DirectSum}
  Let $M$ be a module, $\lambda$ an infinite regular cardinal and
  $\{M_\alpha:\alpha < \lambda\}$ a continuous chain of submodules of
  $M$ such that $M_0=0$ and $M = \bigcup_{\alpha < \lambda}M_\alpha$. Suppose
  that, for each $\alpha < \lambda$, there exists a submodule
  $N_\alpha \leq M$ such that $M_{\alpha+1}=M_\alpha \oplus
  N_{\alpha}$. Then $M = \bigoplus_{\alpha < \lambda}N_\alpha$.
\end{lemma}

We shall denote by $\textrm{Sum}(\mathcal X)$
the class consisting of all direct sums of modules in $\mathcal X$ and
by ${^\perp}\mathcal X$ and $\mathcal X^\perp$ the corresponding
$\textrm{Ext}$-orthogonal classes, i. e., 
\begin{displaymath}
  \mathcal X^\perp = \{Y \in \rmod: \ext_R^1(X,Y) = 0 \quad \forall X \in \mathcal{X}\}
\end{displaymath}
and
\begin{displaymath}
  {^\perp}\mathcal{X} = \{Y \in \rmod: \ext_R^1(Y,X) = 0 \quad \forall X \in
  \mathcal{X}\}.
\end{displaymath}
A {\em cotorsion pair} in $\rmod$ is a pair of classes of modules,
$(\mathcal F, \mathcal C)$, such that $\mathcal F = {}^{\perp}\mathcal
C$ and $\mathcal C = \mathcal F^\perp$. The pair of classes
$({^\perp}(\mathcal X^\perp),\mathcal X^\perp)$ is a cotorsion pair
which is called the cotorsion pair generated by $\mathcal X$.

Let $\mathcal Y$ be a class of right $R$-modules. A left $R$-module
$M$ is said to be $\mathcal Y$-Mittag-Leffler if for each family
$\{Y_i:i \in I\}$ of right $R$-modules belonging to $\mathcal Y$, the
canonical morphism from $\left(\prod_{i \in I}Y_i\right) \otimes_R M$
to $\prod_{i \in I}Y_i \otimes_RM$ is monic. We shall denote by
$\mathcal M^{\mathcal Y}$ the class of all $\mathcal Y$-Mittag-Leffler
left $R$-modules and by $\mathcal{FM}^{\mathcal Y}$ the class of all flat and
$\mathcal Y$-Mittag-Leffler left $R$-modules. When $\mathcal Y=\modr$, we shall call
$\mathcal Y$-Mittag-Leffler modules simply Mittag-Leffler modules.

In \cite[Theorem 2.6]{HerberaTrlifaj}, $\mathcal Y$-Mittag-Leffler
modules are characterized in terms of a local property. We shall use
the following more general definition, which was stated in
\cite[Definition 2.2]{Izurdiaga}.

\begin{definition}
  Let $\kappa$ be an infinite regular cardinal. We shall say that a
  module $M$ is $(\kappa,\mathcal X)$-free if it has a
  $(\kappa,\mathcal X)$-dense system of submodules, that is, a direct
  family of submodules of $M$, $\mathcal S$, satisfying:
  \begin{enumerate}
  \item $\mathcal S \subseteq \mathcal X$;

  \item $\mathcal S$ is closed under well ordered ascending chains of
    length smaller than $\kappa$, and

  \item every subset of $M$ of cardinality smaller than $\kappa$ is
    contained in an element of $\mathcal S$.
  \end{enumerate}
\end{definition}

\begin{remark}
  A similar definition is given in \cite[Definition
2.1]{SlavikTrlifaj}. Note that in our definition the modules in the
dense system need not be $< \kappa$-generated.
\end{remark}

Given an infinite regular cardinal $\kappa$, we shall say that
$\mathcal X$ is closed under $\kappa$-free modules if each
$(\kappa,\mathcal X)$-free module belongs to $\mathcal X$. The
relationship between this local free property and
$\mathcal Y$-Mittag-Leffler, which was proved in \cite[Theorem
2.6]{HerberaTrlifaj}, is that the class of $\mathcal Y$-Mittag-Leffler
is closed under $\aleph_1$-free modules.

\begin{theorem}\label{t:MittagLefflerCharacterization}
  The classes $\mathcal M^{\mathcal Y}$ and $\mathcal{FM}^{\mathcal
    Y}$ are closed under $\aleph_1$-free modules.
\end{theorem}

\section{CALCULATING THE DOUBLE ORTHOGONAL CLASS}
\label{sec:constr-locally-free}


We fix, for the rest of this section, an infinite regular cardinal
$\lambda$, a well ordered system,
$(F_\alpha,g_{\alpha\beta})_{\alpha,\beta < \lambda}$, and an infinite
cardinal $\kappa$ greater or equal than $\lambda$. We shall denote by
$F$ the direct limit of the direct system. Extending the construction
of \cite[\S 3]{SlavikTrlifaj}, we are going to construct an
$\big(\aleph_1,\mathcal G \big)$-free module associated to the direct
system and the cardinal $\kappa$, where
$\mathcal G = \textrm{Sum}(\{F_\alpha:\alpha < \lambda\})$.

We shall denote by $T_\kappa^{<\lambda}$ the set of all sequences of
length smaller than $\lambda$ consisting of elements of $\kappa$, that
is,
\begin{displaymath}
  T_\kappa^{< \lambda} = \{\tau:\mu \rightarrow \kappa | \mu < \lambda\}.
\end{displaymath}
Then $T_\kappa^{<\lambda}$ is a tree whose set of branches,
$\br(T_\kappa^{<\lambda})$, can be identified with the set of all
sequences $\nu:\lambda \rightarrow \kappa$. Given any
$\tau:\mu \rightarrow \kappa$, with $\mu < \kappa$, we shall denote by $\ell(\tau)$ its
length; that is, $\ell(\tau) = \mu$.

For each $\tau \in T_\kappa^{<\lambda}$, denote by
$D_\tau = F_{\ell(\tau)}$ and let
$P = \prod_{\tau \in T_\kappa^{<\lambda}}D_\tau$. Given
$\nu \in \br(T_\kappa^{<\lambda})$ and $\alpha < \lambda$, we define
the morphism $d_{\nu\alpha}:F_\alpha \rightarrow P$ as follows: for
any $x \in F_\alpha$, let $d_{\nu\alpha}(x)$ be the element of $P$
such that:
\begin{itemize}
\item $\big[d_{\nu\alpha}(x)\big](\nu \rest \alpha) = x$;

\item
  $\big[d_{\nu\alpha}(x)\big](\nu \rest \beta) =
  g_{\alpha\beta}(x)$, for each $\beta < \lambda$ with
  $\beta > \alpha$, and

\item $\big[d_{\nu\alpha}(x)\big](\varepsilon) = 0$, for each
  $\varepsilon \in T_\kappa^{<\lambda}$ with
  $\varepsilon \neq \nu \rest \gamma$ for each
  $\alpha \leq \gamma < \lambda$.
\end{itemize}

Denote by $Y_{\nu\alpha}$ the image of $d_{\nu \alpha}$ and note
that, since $d_{\nu\alpha}$ is monic,
$Y_{\nu\alpha} \cong F_\alpha$. Moreover, let
$X_{\nu\alpha} = \sum_{\gamma < \alpha}Y_{\nu\gamma}$. Finally, set
$X_\nu = \sum_{\alpha < \lambda}X_{\nu\alpha}$,
$L = \sum_{\nu \in \br(T_\kappa^{<\lambda})}X_\nu$ and
$D = L \cap \prod^\lambda_{\tau \in T_\kappa^{<\lambda}}D_\tau$. The
following two lemmas state the basic properties of these modules. They
can be proved in a similar fashion to \cite[Lemma 3.5 and
3.6]{SlavikTrlifaj}.

\begin{lemma}
  \begin{enumerate}
  \item For each $\nu \in \br(T_\kappa^{<\lambda})$,
    $X_\nu = \bigoplus_{\alpha < \lambda}Y_{\nu\alpha} \cong \bigoplus_{\alpha < \lambda}F_\alpha$.

  \item $\frac{L}{D} \cong F^{(\br(T_\kappa^{<\lambda}))}$.

  \item $L$ is $(\aleph_1,\mathcal G)$-free, where
    $\mathcal G = \textrm{Sum}(\{F_\alpha:\alpha < \lambda\})$.
  \end{enumerate}
\end{lemma}

\begin{proof}
  (1) By definition,
  $X_{\nu \alpha+1} = X_{\nu \alpha} + Y_{\nu \alpha}$ for each
  $\alpha < \lambda$. We claim that
  this sum is direct. Then the result follows from Lemma
  \ref{l:DirectSum}, since
  $X_\nu = \bigcup_{\alpha < \lambda}X_{\nu\alpha}$.

  Let $\alpha < \lambda$ and $x \in X_{\nu \alpha} \cap Y_{\nu
    \alpha}$. Since $x \in X_{\nu \alpha}$, there exist a finite
  sequence of ordinals smaller than $\alpha$, $\gamma_1, \ldots,
  \gamma_n$ and elements $x_i \in F_{\gamma_i}$ for each $i \in \{1,
  \ldots, n\}$ such that $x = \sum_{i = 1}^n
  d_{\nu\gamma_i}(x_i)$. Then, the element in position $\nu \rest
  \gamma_1$ of $x$ is $x_1$. Since $x \in Y_{\nu\alpha}$, this element has to
  be zero. Thus $x=\sum_{i=2}^nd_{\nu\gamma_i}(x_i)$. Proceeding
  recursively in this way, we can prove that $x_2=x_3 = \cdots = x_n =
  0$ and, consequently, $x=0$.

  (2) First of all, we prove that $\frac{L}{D} = \bigoplus_{\nu \in
    \br(T_\kappa^\lambda)}\frac{X_\nu+D}{D}$. We only have to see
  that the family $\left\{\frac{X_\nu+D}{D}:\nu \in
    \br(T_\kappa^{<\lambda})\right\}$ is independent. Let
  $\nu_0,\nu_1,\ldots,\nu_n \in \br(T_\kappa^{< \lambda})$ be
  distinct sequences. Then there exist $\alpha < \lambda$ such that $\nu_i
  \rest \gamma \neq \nu_j \rest \gamma$ for each $i,j \in \{0,\ldots,
  n\}$ distinct and $\gamma \geq \alpha$. Let $x \in \frac{X_{\nu_0}+D}{D} \cap
  \sum_{i=1}^n\frac{X_{\nu_i}+D}{D}$ and let $y \in X_{\nu_0}$ and
  $z \in \sum_{i =1}^nX_{\nu_i}$ be such that $x = y +D = z+D$. Then
  $y-z \in D$. Since $(y-z)(\nu_0 \rest \gamma) = y(\nu_0 \rest
  \gamma)$ for each $\gamma \geq \alpha$, there exists $\beta < \lambda$ such that $y(\nu_0 \rest
  \gamma) = 0$ for each $\gamma \geq \beta$. This means that $y \in D$
  and, consequently $x = y+D = 0$.

  Now we prove that for each $\nu \in \br(T_\kappa^\lambda)$, $F \cong
  \frac{X_\nu+D}{D}$. Define, for each
  $\alpha < \lambda$, $f_\alpha:F_\alpha \rightarrow
  \frac{X_\nu+D}{D}$ by $f_\alpha (x) = d_{\nu\alpha}(x)+D$ for each
  $x \in F_\alpha$. Then the family $\{f_\alpha:F_\alpha \rightarrow
  \frac{X_\nu+D}{D}\}_{\alpha \in \lambda}$ is a direct system of
  morphisms from $(F_\alpha,g_{\alpha\beta})_{\alpha,\beta < \lambda}$
  since, for each pair of ordinals $\alpha < \beta$, $f_\beta
  g_{\alpha\beta}(x) - f_\alpha (x)$ has support smaller than
  $\lambda$ and, consequently, belongs to $D$. By the universal
  property of the direct limit, the direct system of morphisms, $\{f_\alpha:F_\alpha
  \rightarrow \frac{X_\nu+D}{D}\}_{\alpha \in \lambda}$, induces a
  homomorphism $f:F \rightarrow \frac{X_\nu+D}{D}$ that is easily
  checked to be injective and which is onto by the definition of $X_\nu$.

The conclusion of the proof is that
  \begin{displaymath}
    \frac{L}{D} = \bigoplus_{\nu \in
      \br(T_\kappa^\lambda)}\frac{X_\nu+D}{D} \cong
    F^{(\br(T_\kappa^{<\lambda}))}.
  \end{displaymath}
and we are done.

  (3) Consider the family of submodules
  \[\mathcal S = \{\sum_{\nu \in \Gamma}X_\nu: \Gamma \subseteq
  \br(T_\kappa^{< \lambda}) \textrm{ with } |\Gamma| \textrm{
    countable}\}.\] We are going to see that $\mathcal S$ is a
  $(\aleph_1,\mathcal G)$-dense system.

  Clearly, each countable subset of $L$ is contained in some submodule
  of $\mathcal S$; and $\mathcal S$ is closed under unions of
  countable well ordered chains. 

It remains to prove that each module
  in $\mathcal S$ belongs to $\mathcal G$, i. e., that is a direct sum
  of modules in $\{F_\alpha:\alpha < \lambda\}$. First of all observe
  that, by 1.1, $X_\nu$ can be written as a direct sum of $Y's$ for
  each $\nu \in \textrm{Br}(T_\kappa^{<\lambda})$. Take $\Gamma
  \subseteq \br(T_\kappa^{< \lambda})$ countable, write $\Gamma =
  \{\nu_n:n < \omega\}$ and denote $Z=\sum_{n <
    \omega}X_{\nu_n}$. Define $Z_n = \sum_{m \leq n}X_{\nu_m}$ for each
  $n < \omega$. We claim that, for each $n < \omega$, there exists an
  ordinal $\beta_n < \lambda$ such that $Z_{n+1} = Z_n \bigoplus
  \left(\bigoplus_{\delta > \beta_n}Y_{\nu_{n+1}\delta}\right)$.

  Let $n < \omega$. It is easy to find an ordinal $\beta_n < \lambda$
  and a natural number $k < n+1$ satifying:
  \begin{itemize}
  \item $\nu_{n+1}\rest \beta_n = \nu_k \rest \beta_n$,

  \item and $\nu_{n+1} \rest \alpha \neq \nu_{m} \rest \alpha$ for
    each $\alpha > \beta_n$ and $m < n+1$.
  \end{itemize}
  In order to see that $Z_{n+1} = Z_n \bigoplus
  \left(\bigoplus_{\delta > \beta_n}Y_{\nu_{n+1}\delta}\right)$,
  first note that $Z_n \cap X_{\nu_{n+1}} \leq \prod_{\delta \leq
    \beta_n}D_{\nu_{n+1} \rest \delta}$ by election of $\beta_n$, and
  that $\bigoplus_{\delta > \beta_n}Y_{\nu_{n+1}\delta} \leq
  \prod_{\delta > \beta_n}D_{\nu_{n+1}\rest \delta}$; this implies
  that $Z_n \cap \left(\bigoplus_{\delta >
      \beta_n}Y_{\nu_{n+1}\delta}\right) = 0$.  In order to see that
  $Z_{n+1} = Z_n + \left(\bigoplus_{\delta >
      \beta_n}Y_{\nu_{n+1}\delta}\right)$, we only have to prove that
  $X_{\nu_{n+1}} \leq Z_n + \left(\bigoplus_{\delta >
      \beta_n}Y_{\nu_{n+1}\delta}\right)$. Given $\gamma < \lambda$
  and $x \in F_\gamma$, $d_{\nu_{n+1} \gamma} (x)$ trivially belongs
  to $Z_n + \left(\bigoplus_{\delta >
      \beta_n}Y_{\nu_{n+1}\delta}\right)$ if $\gamma > \beta_n$. If
  $\gamma \leq \beta_n$ then, since $\nu_{n+1}\rest \beta_n =
  \nu_k\rest \beta_n$, we have that
  \[d_{\nu_{n+1} \gamma} (x) - d_{\nu_{n+1}
    \beta_n+1}(g_{\gamma\beta_n+1}(x)) =
  d_{\nu_k\gamma}(x)-d_{\nu_k\beta_n+1}(g_{\gamma\beta_n+1}(x)),\]
  which implies that $d_{\nu_{n+1} \gamma} (x)$ belongs to $Z_n +
  \left(\bigoplus_{\delta > \beta_n}Y_{\nu_{n+1}\delta}\right)$. This
  proves the claim.

  Now, using the claim and Lemma \ref{l:DirectSum} we get that
  \begin{displaymath}
    Z = \bigoplus_{n < \omega}\left(\bigoplus_{\delta >
        \beta_{n-1}}Y_{\nu_{n}\delta}\right)
  \end{displaymath}
  where $\beta_{n-1} < \lambda$. This finishes the proof.
\end{proof}

\begin{remark}
  We shall call $L$ the $\aleph_1$-free module associated to the
  direct system $(F_\alpha,g_{\alpha\beta})_{\alpha<\beta< \lambda}$
  and the infinite cardinal $\kappa$.
\end{remark}

Let $\mathcal X$ be a class of left $R$-modules. We are
going to prove that if $\mathcal X$ is closed under direct sums and
$\aleph_1$-free modules, then ${^\perp}(\mathcal X^\perp)$ contains
all totally ordered direct limits of modules in $\mathcal X$. We begin with a
lemma concerning infinite combinatorics:

\begin{lemma}
  Let $\mu$ be an infinite cardinal and $\xi$ an infinite regular
  cardinal. Then there exists a cardinal $\kappa$ such that
  $\kappa > \mu$, $\kappa^\eta = \kappa$ for each cardinal
  $\eta < \xi$ and $\kappa^\xi = 2^\kappa$.
\end{lemma}

\begin{proof}
  Simply take $\kappa$ to be the suppremum of the increasing sequence
  of cardinals $\{\beta_\alpha:\alpha < \xi\}$ defined recursively as
  follows: $\beta_0 = \mu$; $\beta_{\alpha+1} = 2^{\beta_\alpha}$ for
  each $\alpha < \xi$, and
  $\beta_\alpha = \sup\{\beta_\gamma:\gamma < \alpha\}$ for each
  $\beta < \xi$ limit.
\end{proof}

Now, reasoning as in the last part of the proof of \cite[Theorem
6]{BazzoniStovicek2} we get:

\begin{theorem}\label{t:DoublePerp}
  Let $\mathcal X$ be a class of modules closed under direct sums and
  $\aleph_1$-free modules. Then the class ${^\perp}(\mathcal X^\perp)$ contains
  all totally ordered direct limits of modules in $\mathcal X$.
\end{theorem}

\begin{proof}
  Let $(F_i,f_{ij})_I$ be a direct system of modules in $\mathcal X$
  such that $I$ is totally ordered and denote by $F$ its direct limit. By
  \cite[\S 3, Theorem 36]{Roitman}, there exists a well ordered subset $J$
  of $I$ which is cofinal in $I$. We may assume that $J$ is a
  cardinal, say $\kappa$. Now, if $\cf(\kappa) = \lambda$, there
  exists a cofinal subset $K$ of $I$ of size $\lambda$. In conclusion,
  there exists a well ordered system,
  $(G_\alpha,g_{\alpha\beta})_{\alpha < \beta < \lambda}$, such that
  $G_\alpha \in \mathcal X$ for each $\alpha < \lambda$ and 
  $\displaystyle F=\lim_{\substack{\longrightarrow\\ \alpha <
      \lambda}}G_{\alpha}$.

  Now suppose, in order
  to get a contradiction, that $F \notin {^\perp}(X^\perp)$. Then
  there exists $X \in \mathcal X^\perp$ with $\ext_R^1(F,X) \neq 0$. Take, using the
  preceding lemma, an infinite regular cardinal $\kappa$ satisfying
  $|X| \leq 2^\kappa$, $\kappa^\eta = \kappa$ for each cardinal $\eta
  < \lambda$ and $\kappa^\lambda = 2^\kappa$. Let $L$ be the locally
  $\aleph_1$-free module associated to the direct system
  $(G_\alpha,g_{\alpha\beta})_{\alpha < \beta < \lambda}$ and
  $\kappa$. Since $\mathcal X$ is closed under direct sums and
  $\aleph_1$-free modules, $L$ belongs to $\mathcal X$. Applying $\ext_R^1(\_,X)$ to the short exact sequence
  \begin{displaymath}
    0 \mapright D \mapright L \mapright F^{(\br(T_\kappa^{<\lambda}))}
    \mapright 0
  \end{displaymath}
  we get an epimorphism $\hom(D,X) \rightarrow
  \ext_R^1(F^{(\br(T_\kappa^{<\lambda}))},X)$, which implies that
  $|\hom(D,X)| \leq |\ext_R^1(F^{(\br(T_\kappa^{<\lambda}))},X)|$. But this
  cannot be true since $|\hom(D,X)| \leq 2^\kappa$ (as $|D| = \kappa$)
  and $|\ext_R^1(F^{(\br(T_\kappa^{<\lambda}))},X)| \geq 2^{2^\kappa}$.
\end{proof}

If we specialize Theorem \ref{t:DoublePerp} to the class of flat a
Mittag-Leffler modules, we obtain the following extension of \cite[Theorem
6]{BazzoniStovicek2}.

\begin{corollary}
  Let $\mathcal Y$ be any class of right $R$-modules. Then the class
  ${^\perp}(\mathcal {M^{\mathcal Y}}^{\perp})$
  (resp. ${^\perp}(\mathcal{FM^{\mathcal Y}}^{\perp})$) contains all totally
  ordered direct limits of modules belonging to $\mathcal M^{\mathcal Y}$
  (resp. $\mathcal{FM}^{\mathcal Y}$).
\end{corollary}

\begin{proof}
Follows from theorems \ref{t:DoublePerp} and \ref{t:MittagLefflerCharacterization}.
\end{proof}

Finally, we describe the cotorsion pair generated by the class of all
flat Mittag-Leffler modules for those rings in which all flat modules
are filtered by totally ordered direct limits of projective modules. Of
course, this
class of rings contains all rings in which each flat module is a
totally ordered direct limit of projective modules. Moreover, left perfect
rings, discrete valuation rings and countable rings belong to this
class, see Example \ref{e:example2}. Section 3 is devoted to study
this class of rings.

\begin{corollary} \label{c:CotorsionPairs} Suppose that each flat
  left $R$-module is filtered by totally ordered direct limits of projective
  modules. Then the cotorsion pair generated by $\mathcal{FM}$ is the
  Enochs cotorsion pair $(\mathcal F,\mathcal C)$, where $\mathcal F$
  is the class of all flat modules and $\mathcal C$ is the class of
  all cotorsion modules.
\end{corollary}

\begin{proof}
  By the preceding theorem and Theorem
  \ref{t:MittagLefflerCharacterization}, each totally ordered direct limit of
  projective modules belongs to ${^\perp}(\mathcal{FM}^\perp)$. Since
  ${^\perp}(\mathcal{FM}^\perp)$ is closed under filtrations by
  \cite[Theorem 1.2]{Eklof}, the hypothesis implies that $\mathcal F
  \subseteq {^\perp}(\mathcal{FM}^\perp)$. But this means that
  $\mathcal F={^\perp}(\mathcal{FM}^\perp)$.
\end{proof}

Theorem \ref{t:DoublePerp} is more general and allows us to prove the
analogous of the preceeding result for Mittag-Leffler modules. In this
case, we determine the class ${^\perp}(\mathcal M^\perp)$ for rings in
which each module is filtered by totally ordered direct limits of pure
projective modules.

\begin{corollary}
  Suppose that each left $R$-module is filtered by totally ordered direct limits
    of pure projective modules. Then the cotorsion pair generated by
    $\mathcal M$ is $(\rmod,\mathcal I)$, where $\mathcal I$ is the
    class of all injective modules.
\end{corollary}

\begin{proof}
  Again by the preceding theorem and Theorem
  \ref{t:MittagLefflerCharacterization}, each totally ordered direct limit of
  pure projective modules belongs to ${^\perp}(\mathcal{M}^\perp)$. Since
  ${^\perp}(\mathcal{M}^\perp)$ is closed under filtrations by
  \cite[Theorem 1.2]{Eklof}, the hypothesis implies that $\rmod
  \subseteq {^\perp}(\mathcal{FM}^\perp)$. But this means that
  $\rmod={^\perp}(\mathcal{FM}^\perp)$.
\end{proof}

Left pure semisimple rings trivially satifies the hyphotesis of the
preceeding corollary. Moreover, $\mathbb Z$, by \cite[Corollary
18.4]{Fuchs}, and, more generally, discrete valuation domains satisfy
the hyphotesis of the corollary (this follows form the proof of
Example \ref{e:example2}, where it is
proved that each module over a discrete valuation ring is the union of
a countable chain of modules that are direct sums of cyclics; but
direct sums of cyclics over a discrete valuation ring are
pure-projective). 

\section{TOTALLY ORDERED DIRECT LIMITS OF PROJECTIVE MODULES}
\label{sec:totally-order-direct}

In this section we turn to when all flat modules are filtered by
totally ordered direct limits of projective modules. It is not known which
rings satisfy this property; even, it is not known which rings satisfy
that flat modules are totally ordered direct limits of projectives. We begin
with an example of a flat module which is not a totally ordered direct limit
of projective modules. We shall use the following result which was
communicated by J. \v{S}aroch:

\begin{proposition}\label{p:AlephOmegaFlats}
  Let $M$ be a $\aleph_\omega$-presented flat module which is a
  totally ordered direct limit of projective modules. Then $M$ has finite
  projective dimension.
\end{proposition}

\begin{proof}
  By \cite[\S 3, Theorem 36]{Roitman} there exists a well ordered system of
  projective modules, $(P_\alpha,f_{\alpha\beta})_{\alpha < \beta <
    \kappa}$, whose direct limit is $M$. Denote by $f_\alpha:P_\alpha
  \rightarrow M$ the direct limit canonical map. If $\cf(\kappa) <
  \aleph_\omega$ and $\Gamma$ is a cofinal subset of $\kappa$ of
  cardinality $\cf(\kappa)$, the direct limit of the system
  $(P_\alpha,f_{\alpha\beta})_{\alpha,\beta \in \Gamma}$ is $M$. By
  \cite[Theorem 2.3]{Osofsky}, $M$ has finite projective dimension.

  So suppose that $\cf(\kappa) > \aleph_\omega$. We are going to
  construct a countable chain $\beta_0 < \beta_1 < \cdots$ in $\kappa$
  and $\leq \aleph_\omega$-generated direct summands $S_n$ of
  $P_{\beta_n}$ satisfying $f_{\beta_0}(S_0) = M$,
  $f_{\beta_n\beta_{n+1}}(S_n) \leq S_{n+1}$ and $\ker
  f_{\beta_n\beta_{n+1}} \cap S_n = \ker f_{\beta_n} \cap S_n$ for
  each $n < \omega$. With this construction made, setting $h_{nn+1} =
  f_{\beta_n\beta_{n+1}} \rest S_n$ for each $n < \omega$ (and
  $h_{nm}$ the corresponding compositions for each $n < m < \omega$),
  we get a direct system of projective modules,
  $(S_n,h_{nm})_{n<m<\omega}$, whose direct limit is $M$. Then, again by
  \cite[Theorem 2.3]{Osofsky}, $M$ has finite projective dimension.

  So, it only remains to make the construction; we shall proceed by
  induction on $n$. Suppose that $n=0$. Since $M$ is
  $\aleph_\omega$-generated and $\cf(\kappa) > \aleph_\omega$, there
  exists $\beta_0 < \kappa$ such that $f_{\beta_0}(P_{\beta_0}) =
  M$. Since $P_{\beta_0}$ is a direct sum of countably generated
  modules and $M$ is $\aleph_\omega$-generated, there exists a
  $\aleph_\omega$-generated direct summand $S_0$ of $P_{\beta_0}$ such
  that $f_{\beta_0}(S_0) = M$.

  Now assume that we have made the construction for some $n<\omega$,
  and let us do it for $n+1$. Since $M$ is $\aleph_\omega$-presented
  and $f_{\beta_n}(S_n) = M$, $\ker f_{\beta_n} \cap S_n$ has a
  generating set $X$ of cardinality less or equal than
  $\aleph_\omega$. Now, for each $x \in X$, as $f_{\beta_n}(x) = 0$,
  there exists $\beta_x < \kappa$ with $\beta_x \geq \beta_n$ such that $f_{\beta_n\beta_x}(x) =
  0$. Then, taking $\beta_{n+1}$ the suppremun of $\{\beta_x: x \in
  X\}$, we get an ordinal smaller than $\kappa$ (remember:
  $\cf(\kappa) > \aleph_\omega$) with the property that $\ker
  f_{\beta_n \beta_{n+1}} \cap S_n = \ker f_{\beta_n} \cap
  S_n$. Finally, in order to finish the inductive step, take $S_{n+1}$
  an $\aleph_\omega$-generated direct summand of $P_{\beta_{n+1}}$
  containing $f_{\beta_n\beta_{n+1}}(S_n)$.
\end{proof}

Now we can give the announced example:

\begin{example}\label{e:example}
  Let $R$ be the boolean algebra freely generated (as a boolean
  algebra) by a set of $X$ of cardinality $\aleph_\omega$. Then, by
  \cite[pp. 43]{Sikorski}, $X$ is a nice set of idempotents in the
  sense of \cite{Osofsky1970}. The left ideal $I_X$ generated by $X$
  is flat, as $R$ is von-Neumann regular, $\aleph_\omega$-presented
  (look at the projective presentation of $I_X$ constructed in
  \cite[p. 642]{Osofsky1970}) and has infinite projective dimension by
  \cite[Theorem A]{Osofsky1970}. By the preceding result, $I_X$ cannot
  be a totally ordered direct limit of projective modules.
\end{example}

Now we give examples of rings satisfying that flat modules are
filtered by totally ordered direct limits of projectives.

\begin{examples}\label{e:example2}
  \begin{enumerate}
  \item The following rings satisfy that all flat modules are totally
    ordered direct limits of projective modules:
    \begin{enumerate}
    \item Left perfect rings.

    \item $\mathbb Z$.

    \item Discrete valuation domains.
    \end{enumerate}

  \item If $R$ is countable, the every flat module is filtered by
    totally ordered direct limits of projective modules.
  \end{enumerate}

  As a consequence of Corollary \ref{c:CotorsionPairs}, in each of
  these rings the cotorsion pair generated by the flat and Mittag
  Leffler modules is the Enochs cotorsion pair.
\end{examples}

\begin{proof}
  We begin with the proof of (1). Part (a) is trivial; (b) follows
  from \cite[Corollary 18.4]{Fuchs}. Part (c) can be proved in a
  similar way as \cite[Proposition 18.3 and Corollary 18.4]{Fuchs}
  with some arithmetic modifications; we sketch the proof for
  completeness.

  Suppose that $R$ is a discrete valuation domain with prime element
  $p$. Recall that for each $r \in R$ there exists $n \in \omega$ and
  units $u,v \in R$ such that $r = up^n = p^nv$; consequently, there
  exists $s \in R$ such that $pr = sp$. In what follows, if $M$ is a
  left $R$-module, $\{m_i:i \in I\}$ is a subset of $M$ and $\sum_{i
    \in I}r_im_i$ is a linear combination, the set $\{i \in I:r_i \neq
  0\}$ is suppose to be finite. We refer to \cite{KrylovTuganbaev} for
  facts about discrete valuation rings.

  First of all, we prove that if $M$ is a left $R$-module such that
  $pM$ is a direct sum of cyclic modules, then $M$ is a direct sum of
  cyclic modules too. Suppose that $pM= \bigoplus_{i \in I}Rn_i$ for a
  family of elements $\{n_i:i \in I\} \subseteq pM$ and write, for
  each $i \in I$, $n_i = pm_i$ for some $m_i \in M$. Then $S'=\{m_i: i
  \in I\}$ is an independent set since, if we take a linear
  combination $\sum_{i \in I}r_im_i$ which is equal to zero, then
  $\sum_{i \in I}pr_im_i = 0$. But, for each $i \in I$, $pr_i = s_ip$
  for some $s_i \in R$, and, consequently, $\sum_{i \in I}s_in_i = 0$;
  since $\{n_i: i \in I\}$ is independent, $s_i = 0$ for each $i \in
  I$, which implies that $r_i = 0$ for each $i \in I$. Now, extend
  this independent set to a maximal one, say $S = \{m_i: i \in I\}
  \cup \{n_j:j \in J\}$. Then $S$ generates $M$ since, given any $m
  \in M$, $pm = \sum_{i \in I}r_ipm_i$ for a family $\{r_i:i \in I\}
  \subseteq R$; writing, for each $i \in I$, $r_ip = ps_i$ for some
  $s_i \in R$, we get that $p\left(m-\sum_{i \in I}s_im_i\right)=0$
  and, consequently, $m-\sum_{i \in I}s_im_i$ belongs to the socle of
  $M$, which is contained in the submodule generated by $S$. Then $m$
  is generated by $S$ and we are done.

  Now, define the chain of submodules of $M$, $\{M_n:n < \omega\}$,
  recursively as follows: $M_0$ is the submodule generated by a
  maximal independent subset of $M$; if $M_n$ has been defined for
  some $n < \omega$, set $M_{n+1} = \{m \in M:pm \in M_n\}$. Note
  that, for each $n < \omega$, $M_n$ is a direct sum of cyclics by
  the previous result. It remains to prove that $M = \bigcup_{n <
    \omega}M_n$. Suppose that this is not true and let $m \in M -
  \bigcup_{n < \omega}M_n$; then $p^nm \notin M_0$ for each $n \in
  \omega$, which implies that $S \cup \{m\}$ is linearly independent,
  since for any linear combination, $\sum_{i \in I}r_im_i + rm$, if
  $\sum_{i \in I}r_im_i + rm = 0$, writing $r = up^nm$ for some unit
  $u \in R$ and $n < \omega$, we get $p^nm = \sum_{i \in I}-ur_im_i
  \in M_0$, which is a contradiction.

  In order to derive (2) note that, as it is proved in
  \cite[Proposition 7.4.3]{EnochsJenda}, if $R$ is countable then
  every flat module is filtered by countably generated flat
  modules. Since the ring is countable, these modules are actually
  countably presented. By \cite[Theorem]{Osofsky2}, each countably
  presented flat module is the direct limit of a countable system of
  projective modules; but it is easy to see that each countable direct
  poset has a well ordered cofinal subset. The conclusion is that each
  flat module is filtered by totally ordered direct limits of projective
  modules.
\end{proof}

Now if we assume the hyphotesis that there do not exist
inaccessible cardinals, we can prove that all totally ordered direct limits of
projective modules have finite projective dimension. We start with a
theorem in ZFC.

\begin{theorem}
  Let $M$ be a flat module. Suppose that $M$ is the direct limit of a
  well ordered system of projective modules, $(P_\alpha,f_{\alpha
    \beta})_{\alpha < \beta < \lambda}$, for some cardinal $\lambda$
  smaller than the first inaccessible cardinal. Then $M$ has finite
  projective dimension.
\end{theorem}

\begin{proof}
  We shall prove that for each infinite cardinal $\lambda$ smaller
  than the first inaccesible cardinal, there exists a natural number
  $n_\lambda$ such that the direct limit of each direct system
  $(P_\alpha,f_{\alpha\beta})_{\alpha < \beta < \mu}$, with $\mu \leq
  \lambda$, consisting of projective modules has projective dimension
  less or equal than $n_\lambda$. We shall proceed inductively on
  $\lambda$.

  If $\lambda = \aleph_0$ the result follows from \cite[Theorem
  2.3]{Osofsky} (actually from \cite{Berstein}). Suppose that $\lambda
  > \aleph_0$ and that we have proven the claim for each direct system
  of cardinality less than $\lambda$; let
  $(P_\alpha,f_{\alpha\beta})_{\alpha < \beta < \lambda}$ be a direct
  system consisting of projective modules. If $\lambda$ is a direct limit
  cardinal, $\cf(\lambda)$ is smaller than $\lambda$ by
  hyphotesis. Since $\lambda$ has a cofinal subset of cardinality
  $\cf(\lambda)$, the induction hyphotesis implies that $\displaystyle
  \lim_{\longrightarrow}P_\alpha$ has projective dimension less or
  equal than $n_{\cf(\lambda)}$. Thus, the result is proven taken $n_\lambda
  = n_{\cf(\lambda)}$.

  It remains to prove case $\lambda$ successor. Assume that $\lambda =
  \mu^+$ for some infinite cardinal $\mu$. Denote, for each
  ordinal $\alpha < \lambda$, by $\displaystyle
  G_\alpha = \lim_{\substack{\longrightarrow\\\gamma < \alpha}}
  P_\gamma$ and, for any other ordinal $\beta > \alpha$, by $g_{\alpha
    \beta}:G_\alpha \rightarrow G_\beta$ the induced map by the
  structural maps of the corresponding direct limits. We obtain, in
  this way, a direct system $(G_\alpha,g_{\alpha\beta})_{\alpha <
    \beta < \lambda}$ whose direct limit is $\displaystyle
  \lim_{\longrightarrow}P_\alpha$ and such that the projective
  dimension of $G_\alpha$ is less or equal than $n_\mu$ (by induction
  hyphotesis). Now consider the canonical presentation of
  $\displaystyle \lim_{\longrightarrow}P_\alpha$,
  \begin{equation}\label{eq:1}
    0 \mapright X \mapright \bigoplus_{\alpha < \lambda}P_\alpha \mapright
    \lim_{\longrightarrow}P_\alpha \mapright 0
  \end{equation}
  and, for each $\alpha < \lambda$, the corresponding one of
  $G_\alpha$,
  \begin{displaymath}
    0 \mapright X_\alpha \mapright \bigoplus_{\gamma < \alpha}P_\gamma
    \mapright G_\alpha \mapright 0
  \end{displaymath}
  It is very easy to see that $X = \bigcup_{\alpha < \lambda}
  X_\alpha$ and that this union is continuous. Moreover, since
  $X_\alpha$ has projective dimension less or equal than $n_\mu-1$,
  $X$ has projective dimension less or equal than $n_\mu$ by
  \cite[Corollary 0.3]{Osofsky}. Then, looking at the short exact
  sequence (\ref{eq:1}), we conclude that $\displaystyle
  \lim_{\longrightarrow}P_\alpha$ has projective dimension less or
  equal than $n_\mu+1$. Thus, the result is proved setting $n_\lambda
  = n_\mu+1$.
\end{proof}

As a consequence of this result, if we assume that there do not exist
inaccessible cardinals we get:

\begin{corollary}
  Assume that there do not exist inaccessible cardinals. Then each
  flat module which is a totally ordered direct limit of projective modules
  has finite projective dimension.
\end{corollary}

\begin{proof}
  Simply note that each totally ordered set has a cofinal well ordered
  subset (see \cite[\S 3, Theorem 36]{Roitman}). Then the proof of the
  preceding theorem works for this one.
\end{proof}

Recall that the ring $R$ (see \cite{EnochsJendaRamos}) is said to be
$n$-perfect if each flat module has projective dimension less or equal
than $n$. As a consequence of the last result we get:

\begin{corollary}\label{c:n-perfect}
  Assume that there do not exist inaccessible cardinals. If each flat
  module is filtered by totally ordered direct limits of projective
  modules, then there exists a natural number $n$ such that $R$ is
  left $n$-perfect.
\end{corollary}

\begin{proof}
  By the preceding corollary, each module which is a totally ordered
  direct limit of projective modules has finite projective
  dimension. By the hyphotesis and Auslander's Theorem
  (\cite[Proposition 3]{Auslander}), each flat module has finite
  projective dimension. Now, it is easy to see that there exists a
  natural number $n$ such that each flat module has projective
  dimension less or equal than $n$: simply use \cite[Lemma
  3.9]{Izurdiaga} taking $\mathcal X$ the class of all projective
  modules and $\mathcal Y$ the class of all totally ordered direct limits of
  projective modules, which is clearly closed under countably direct
  sums.
\end{proof}

\end{document}